\documentclass[preprint,11pt,3p]{elsarticle}
\usepackage{times}
\usepackage{color}
\usepackage[colorlinks=true]{hyperref}
\usepackage{amsmath,amssymb,amsthm,amscd,mathrsfs}
\newtheorem{theorem}{Theorem}[section]

\newtheorem{proposition}{Proposition}[section]
\newtheorem{lemma}{Lemma}[section]

\newtheorem{corollary}{Corollary}[section]

\usepackage{enumerate}
\usepackage{tikz,pgfplots}

\RequirePackage[physics]{physicx}
\numberwithin{equation}{section}
\newcommand{\N}{\mathbb{N}}
\newcommand{\R}{\mathbb{R}}

\newcommand{\x}{$x$ in $(0,1] $}

\newcommand{\psiassert}{\psi_{\omega_1}(\pi(\sigma\omega))}
\newcommand{\fassert}{f_{\omega_1}(\pi(\sigma\omega))}
\newcommand{\diff}{\, \mathrm{d}}

\biboptions{sort&compress}
\journal{Elsevier}

\begin{document}

\begin{frontmatter}

\title{ Multifractal analysis of the power-2-decaying Gauss-like expansion}
\author[label1]{XUE-JIAO WANG\corref{cor1}}
\ead{xuejiao.wang@whu.edu.cn}

\address[label1]{School of Mathematics and Statistics, Wuhan University, Wuhan 430072, China}

\cortext[cor1]{corresponding author}

\begin{abstract}
 Each real number $x\in[0,1]$ admits a unique power-2-decaying Gauss-like expansion (P2GLE for short) as $x=\sum_{i\in\mathbb{N}} 2^{-(d_1(x)+d_2(x)+\cdots+d_i(x))}$, where $d_i(x)\in\mathbb{N}$. For any $x\in(0,1]$, the Khintchine exponent $\gamma(x)$ is defined by $\gamma(x):=\lim_{n\to\infty}\frac{1}{n}\sum_{j=1}^nd_j(x)$ if the limit exists. We investigate the sizes of the level sets $E(\xi):=\{x\in(0,1]:\gamma(x)=\xi\}$ for $\xi\geq 1$. Utilizing the Ruelle operator theory, we obtain the Khintchine spectrum $\xi\mapsto\dim_H E(\xi)$, where $\dim_H$ denotes the Hausdorff dimension. We establish the remarkable fact that the Khintchine spectrum has exactly one inflection point, which was never proved for the corresponding spectrum in continued fractions. As a direct consequence, we also obtain the Lyapunov spectrum. Furthermore, we find the Hausdorff dimensions of the level sets $\{x\in(0,1]:\lim_{n\to\infty}\frac{1}{n}\sum_{j=1}^{n}\log(d_j(x))=\xi\}$ and $\{x\in(0,1]:\lim_{n\to\infty}\frac{1}{n}\sum_{j=1}^{n}2^{d_j(x)}=\xi\}$.
\end{abstract}

\end{frontmatter}

\section{Introduction and main results}
Let $\N=\{1,2,\dots\}$ denote the set of positive integers. The power-2-decaying Gauss-like expansion (P2GLE for short) of a real number can be generated by the transformation $T: (0,1] \to (0,1]$ defined by
\begin{equation}\label{expansion}
  Tx:=2^nx-1 \quad\text{for } x\in(1/2^n,1/2^{n-1}] ,\quad n\in\N.
\end{equation}
By \cite[Proposition 2.1]{neunhauserer2011hausdorff}, we know that every real number $x$ in $(0,1]$ can be written uniquely as the form
\begin{equation}\label{ex}
  x=\sum_{i=1}^{\infty}2^{-(d_1(x)+d_2(x)+\cdots+d_i(x))}
\end{equation}
where $d_1(x)=\lceil-\log_2 x\rceil$ and $d_i(x)=d_1(T^{i-1}x)$ for $i\geq 2$. Here, we use $\lceil x\rceil$ to denote the smallest integer greater than or equal to $x.$
For any $d_1,\cdots,d_n\in\N$, the $n$th-order cylinder is defined by
\[I_n(d_1,\cdots,d_n):=\{x\in(0,1]:d_1(x)=d_1,\cdots,d_n(x)=d_n\}.\]


\medskip
Let $\varphi:[0,1]\to \R$ be a real-valued function (usually called potential). For any $n\geq 1$, denote the $n$th Birkhoff sum of $\varphi$ by $$S_n\varphi(x):=\sum_{j=0}^{n-1}\varphi(T^jx),\quad{x\in[0,1]}.$$ The Birkhoff average of $\varphi$ is defined as 
\[\bar {\varphi}(x):=\lim_{n\to\infty}\frac{1}{n}S_n\varphi(x),\quad x\in(0,1]\](if the limit exists). In multifractal analysis, we are interested in the Hausdorff dimension of the level sets
\begin{equation}\label{star}
E_{\varphi}(\xi)=\{x\in(0,1]:\bar {\varphi}(x)=\xi\}, \quad \xi\in\R. 
\end{equation}
The so-called Birkhoff spectrum is the function $t_{\varphi}:\R\to[0,1]\cup \{-\infty\}$ defined by
\[t_{\varphi}(\xi):=\dim_H E_{\varphi}(\xi)\]
where $\dim_H E_{\varphi}(\xi)$ denotes the Hausdorff dimension of $E_{\varphi}(\xi).$ By convention, $\dim_H\emptyset=0.$
\medskip

Some metric theorems of P2GLE have been obtained in \cite{neunhauserer2011hausdorff,Jorg,araising}. In particular, in \cite{Jorg} we know that the transformation $T$ is measure-preserving and ergodic with respect to the Lebesgue measure.  In \cite{neunhauserer2011hausdorff}, the author showed that the Hausdorff dimension of the set of numbers with prescribed digits in $A\subset\N$ in their P2GLE is the unique real solution of $\sum_{n\in A}2^{-ns}=1$, and hence the set $$\{x\in(0,1]:d_i(x)\text{ is bounded}\}$$ has Hausdorff dimension one. It is known to \cite{araising} that  $$\dim_H\{x\in(0,1]:\lim_{i\to\infty}d_i(x)=\infty\}=0.$$ 

\medskip
In this paper, we aim at finding the Birkhoff spectrum $\xi\mapsto\dim_H E_{\varphi}(\xi)$ for different choices of potential $\varphi$ for P2GLE. For any $x\in(0,1]$ with P2GLE \eqref{ex}, we define its Khintchine exponent $\gamma(x)$ and Lyapunov exponent $\lambda(x)$ respectively by 
     \begin{align*}
      \gamma(x):&=\lim_{n\to\infty}\frac{1}{n}\sum_{j=1}^nd_j(x)=\lim_{n\to\infty}\frac{1}{n}\sum_{j=0}^{n-1}d_1(T^jx),\\
      \lambda(x):&=\lim_{n\to\infty}\frac{1}{n}\log\abs{(T^n)^{\prime}(x)}=\lim_{n\to\infty}\frac{1}{n}\sum_{j=0}^{n-1}\log\abs{T'(T^jx)},
     \end{align*}
     if the limits exist.
     By definition, the Khintchine exponent is the Birkhoff average of  the potential $x\mapsto d_1(x),$ and the Lyapunov exponent is the Birkhoff average of the potential $x\mapsto \log\abs{T'(x)}$.
  
    \medskip
     Since the transformation $T$ is ergodic with respect to the Lebesgue measure, it follows that for Lebesgue almost all \x,
     \begin{align*}
      \gamma(x)&=\int\lceil-\log_2 x\rceil \,\mathrm{d}x=\sum_{n=1}^{\infty}n2^{-n}=2,\\
      \lambda(x)&=\int\log\abs{T'(x)}\,\mathrm{d}x=(\log2)\sum_{n=1}^{\infty}n2^{-n}=2\log2.
     \end{align*}
     \medskip

Note that $\gamma(x)\geq 1$ and $\lambda(x)\geq\log2$ for all \x. We study the level sets:    
\[E(\xi):=\{x\in(0,1]:\gamma(x)=\xi\},\quad \xi\geq 1,\]
 and
\begin{equation*}
  F(\beta):=\{x\in(0,1]:\lambda(x)=\beta\},\quad \beta\geq\log2.
\end{equation*}
\medskip
 The Khintchine spectrum and Lyapunov spectrum are defined by 
 \begin{align*}
  t(\xi):&=\dim_H E(\xi),\quad\xi\geq 1,\\
   \tilde{t}(\beta):&=\dim_H F(\beta),\quad\beta\geq\log2.
 \end{align*}
 
 \medskip
     Fan, Liao, Wang and Wu (\cite{A.H.Fan}) obtained complete graphs of the Khintchine spectrum and Lyapunov spectrum of the Gauss map associated to continued fractions. Applying the method of \cite{A.H.Fan}, we study the Birkhoff spectra of P2GLE. The major tool is the Ruelle operator theory.  For a given potential $\varphi$, we define the pressure function by
     \begin{equation}\label{equa.0}
     P(t,q)=P_{\varphi}(t,q):=\lim_{n\to\infty}\frac{1}{n}\log\sum_{\omega_1\in\N}\sum_{\omega_2\in\N}\cdots\sum_{\omega_n\in\N}\exp(\sup_{\substack{d_i(x)=\omega_i\\ 1\leq i\leq n}}S_n(-t\log\abs{T'(x)}+q\varphi(x))).
     \end{equation}
     We will prove (Theorem \ref{them.dim}) that $\dim_H E_{\varphi}(\xi)$ is exactly the first coordinate $t(\xi)$ of the unique solution $\big(t(\xi),q(\xi)\big)$ of the system
   \begin{equation}\label{equa.pressure}
    \begin{cases}
     P(t,q)=q\xi,\\
     \dfrac{\partial P}{\partial q}(t,q)=\xi.
       \end{cases}
     \end{equation}
      The notation $P_{\varphi}(t,q)$ and $E_{\varphi}(\xi)$ will be consistently used through this paper, and their specific forms are determined by the potential $\varphi$.

      \medskip
We obtain the following results of Khintchine spectrum for P2GLE.

\begin{theorem}\label{thm.1}
   We have $\dim_H E(1)=0$ and for any $\xi>1$,
   \begin{equation}\label{spec.1}
    t(\xi)=\dim_H E(\xi)=\dfrac{\frac{\log(\xi-1)}{\xi}+\log\xi-\log(\xi-1)}{\log2}.
   \end{equation}
    Furthermore, the Khintchine spectrum $t:\xi\mapsto\dim_H E(\xi)$  has the following properties:
  \begin{item}
    (1) $t(2)=1$;
  \end{item}
  \begin{item}
  (2) $\lim_{\xi\to 1} t(\xi)=0$ and $\lim_{\xi\to \infty} t(\xi)=0$;
  \end{item}
  \begin{item}
    (3) $t^{\prime}(\xi)>0$ for $\xi<2$, $t^{\prime}(2)=0$, and $t^{\prime}(\xi)<0$ for $\xi>2$;
  \end{item}

  \begin{item}
    (4) $\lim_{\xi\to 1} t'(\xi)=\infty$ and $\lim_{\xi\to \infty} t'(\xi)=0$;
  \end{item}
  \begin{item}
(5) the function $t(\xi)$ has exactly one inflection point.
  \end{item}
\end{theorem}
\medskip
We remark that the last point of Theorem \ref{thm.1} was not proved for continued fraction expansion in \cite{A.H.Fan}. The reason is that in our P2GLE case, we have an explicit formula for the pressure function $P(t,q)$, while in the continued fraction expansion case such formula does not exist.

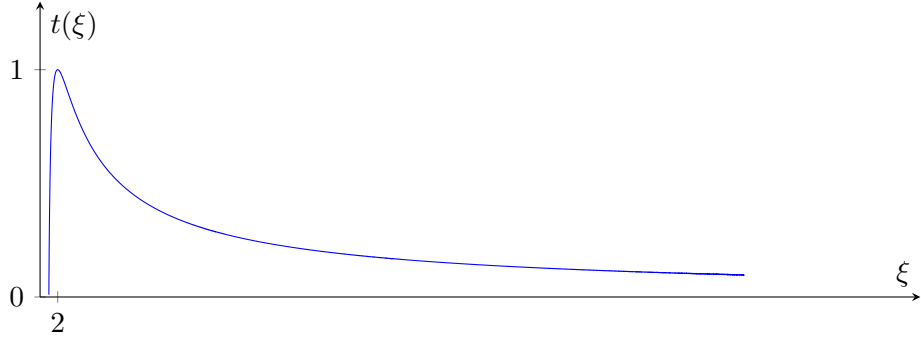
\begin{figure}[t]
  \centering
\begin{tikzpicture}
  \begin{axis}[
      axis lines=middle,
		  inner axis line style={=>},
      xlabel=$\xi$,
      ylabel=$t(\xi)$,
      xmin=0, xmax=100,
      ymin=0, ymax=1.3,
      axis lines=left,
      width=0.8\textwidth,
      height=5.5cm,
      ytick={0,1},   
      xtick={2}    
  ]
  \addplot[domain=-1:80, samples=10000, color=blue]{(ln(x-1)/x+ln(x)-ln(x-1))/ln(2)};
  \end{axis}
  \end{tikzpicture}
  \caption{Khintchine spectrum $\big(\varphi(x)=d_1(x)\big)$}
  \label{fig:1}
\end{figure}
\medskip
The Lyapunov spectrum gains in interest if we realize that  $$\lim_{n\to\infty}\frac{1}{n}\sum_{j=0}^{n-1}\log\abs{T'(T^jx)}=(\log2)\cdot\lim_{n\to\infty}\frac{1}{n}\sum_{j=1}^nd_j(x).$$
Therefore, replacing $\xi$ by $\frac{\beta}{\log2}$ in \eqref{spec.1}, we obtain the Lyapunov spectrum. 

\begin{corollary}\label{thm.2}
   We have $\dim_HF(\log2)=0$ and for any $\beta>\log2$,
\begin{equation}
  \tilde{t}(\beta)=\dim_H F(\beta)=\frac{1}{\beta}\log(\dfrac{\beta}{\log2}-1)-\frac{1}{\log2}\log(1-\dfrac{\log2}{\beta}).
\end{equation}
   Furthermore, the Lyapunov spectrum $\tilde{t}:\beta\mapsto \tilde{t}(\beta)$ has the following properties:
  \begin{item}
    (1) $\tilde{t}(2\log2)=1$;
  \end{item}
  \begin{item}
    (2) $\lim_{\beta\to \log2} \tilde{t}(\beta)=0$ and $\lim_{\beta\to \infty} \tilde{t}(\beta)=0$;
  \end{item}
  \begin{item}
    (3) $\tilde{t}^{\prime}(\beta)>0$ for $\beta<2\log2$, $\tilde{t}^{\prime}(2\log2)=0$, and $\tilde{t}^{\prime}(\beta)<0$ for $\beta>2\log2$;
  \end{item}
  \begin{item}
    (4) $\lim_{\beta\to \log2} \tilde{t}^{\prime}(\beta)=\infty$ and $\lim_{\beta\to \infty} \tilde{t}^{\prime}(\beta)=0$;
  \end{item}

  \begin{item}
(5) the function $\tilde{t}(\beta)$ has exactly one inflection point.
  \end{item}

\end{corollary}
\begin{figure}[t]
  \centering
 \begin{tikzpicture}
  \begin{axis}[
      axis lines=middle,
      xlabel=$\beta$,
      ylabel=$\tilde{t}(\beta)$,
      xmin=0, xmax=100,
      ymin=0, ymax=1.3,
      axis lines=left,
      width=0.8\textwidth,
      height=5.5cm,
      ytick={0,1}, 
      xtick={1.5},
      xticklabels={$2\log2$},                 
  ]
  \addplot[domain=-1:80, samples=10000, color=blue]{(1 / x) *ln(x/ln(2)-1)-(1/ln(2))*ln(1-ln(2)/x) };
  \end{axis}

  \end{tikzpicture}
  \caption{Lyapunov spectrum $\big(\varphi(x)=\log\abs{T'(x)}\big)$}
  \label{fig:2}
\end{figure}
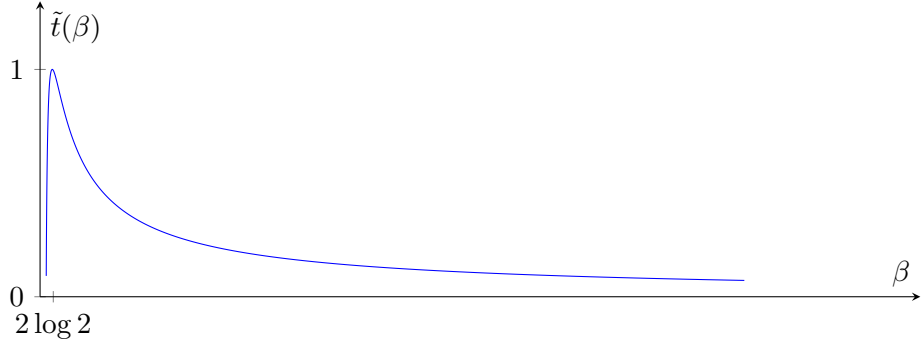
\medskip
We illustrate the Birkhoff spectra in Theorem \ref{thm.1} and Corollary \ref{thm.2} by Figures \ref{fig:1} and \ref{fig:2}.
\medskip

The same method also works for general potential $\varphi$. To illustrate some different Birkhoff spectra, we study the special cases of $\varphi(x)=\log(d_1(x))$ and $\varphi(x)=2^{d_1(x)}$. From the Birkhoff ergodic theorem, we deduce that for Lebesgue almost all \x,
 $$\lim_{n\to\infty}\frac{1}{n}\sum_{j=0}^{n-1}\log(d_1(T^jx))=\xi_0=\int\log\lceil-\log_2 x\rceil \,\mathrm{d}x=\sum_{n=1}^{\infty}2^{-n}\log n\approx 0.507834,$$
 and 
$$\lim_{n\to\infty}\frac{1}{n}\sum_{j=0}^{n-1}2^{d_1(T^jx)}=\int 2^{\lceil-\log_2 x\rceil} \,\mathrm{d}x=\infty.$$
Now let us state our main results for $\varphi(x)=\log(d_1(x))$ and $\varphi(x)=2^{d_1(x)}$.

\begin{theorem}\label{thm.3}
  Let $\varphi:(0,1]\to\R$ be the potential defined by $\varphi(x)=\log(d_1(x))$, with $d_1(x)=\lceil-\log_2 x\rceil$, and $E_{\varphi}(\xi)$ be defined as in \eqref{star}. Then, we have that the pressure function $P$ defined in \eqref{equa.0} is  $$P(t,q)=P_{\varphi}(t,q)=\log(\sum_{n=1}^{\infty}2^{-nt}\cdot n^q),$$ and that for any $\xi>0$, there exists a unique solution $\big(t(\xi),q(\xi)\big)$ to the system \eqref{equa.pressure}. Moreover, $\dim_H E_{\varphi}(0)=0$ and for any $\xi>0$,
$\dim_H E_{\varphi}(\xi)=t(\xi)$. 
Let $\xi_0=\int_0^1\log\lceil-\log_2 x\rceil \,\mathrm{d}x$. The Birkhoff spectrum $t:\xi\mapsto t(\xi)$ has the following properties:
\begin{item}
  (1) $t(\xi_0)=1;$
\end{item}
\begin{item}
  (2) $t^{\prime}(\xi)>0$ for $\xi<\xi_0$, $t^{\prime}(\xi_0)=0$, and $t^{\prime}(\xi)<0$ for $\xi>\xi_0;$ 
\end{item}
\begin{item}
  (3) $\lim_{\xi\to0}t(\xi)=0$ and $
  \lim_{\xi\to\infty}t(\xi)=0; $
\end{item}
\begin{item}
  (4) $\lim_{\xi\to0}t^{\prime}(\xi)=\infty;$
\end{item}
\begin{item}
  (5) $t^{\prime\prime}(\xi_0)<0$ and there exits $\xi_1>\xi_0$ such that $t^{\prime\prime}(\xi_1)>0$, so the function $t(\xi)$ is neither convex nor concave. 
\end{item}
\end{theorem}
\begin{theorem} \label{thm.4}
  \medskip
Let $\varphi(x)=2^{d_1(x)}$ for $x\in(0,1]$. Then the pressure function $P$ defined in \eqref{equa.0} is
\[P(t,q)=P_{\varphi}(t,q)=\log(\sum_{n=1}^{\infty}2^{-nt}e^{2^nq}).\] 
We have that $\dim_H E_{\varphi}(2)=0$ and for any $\xi>2$, $\dim_HE_{\varphi}(\xi)=t(\xi)$, which is the first coordinate of the unique solution of the equation \eqref{equa.pressure}. The Birkhoff spectrum $t:\xi\mapsto t(\xi)$ has the following properties:
        \begin{item}
          (1) $\dim_H E_{\varphi}(\infty)=1;$
        \end{item}
        \begin{item}
          (2) $t^{\prime}(\xi)> 0$ for $\xi>2$;
        \end{item}
        \begin{item}
          (3) $\lim_{\xi\to2}t(\xi)=0$ and $ \lim_{\xi\to\infty}t(\xi)=1;$
        \end{item}
        \begin{item}
          (4) $\lim_{\xi\to2}t'(\xi)=\infty$ and $\lim_{\xi\to\infty}t'(\xi)=0;$
        \end{item}
      \end{theorem}




  \medskip
 An outline of this paper is as follows. In Section 2, we review the Ruelle operator theory and the method of pressure function. In Sections 3 and 4, we prove the results for Khintchine and Lyapunov spectra (Theorem \ref{thm.1} and Corollary \ref{thm.2}). Sections 5 and 6 are devoted to proving Theorems \ref{thm.3} and \ref{thm.4} respectively.

 \section{Conformal iterated function systems and Ruelle operator theory}
 \subsection{Conformal iterated function systems}
 From now on, $X$ stands for a non-empty compact connected subset of $\R^d$ equipped with a metric $\rho$. Let $S=\{\phi_i: X\to X:i\in\N\}$ be an iterated function system for which there exists $0<s<1$ such that for each $i\in\N$ and all $x,y\in X$,
 \begin{equation}
   \rho(\phi_i(x),\phi_i(y))\leq s\rho(x,y).
 \end{equation}  
 We follow the notation of \cite{A.H.Fan}.
 \begin{itemize}
   \item $\N^n:=\{\omega:\omega=(\omega_1,\cdots\omega_n),\omega_k\in\N,1\leq k\leq n\}.$
   \item $\N^{\infty}:=\prod _{i=1}^{\infty}\N$.
   \item $[\omega|_n]=[\omega_1\cdots\omega_n]=\{x\in\N^{\infty}:x_1=\omega_1,\cdots x_n=\omega_n\}$.
   \item $\abs{\omega}$ denotes the length of $\omega.$
   \item $\phi_{\omega}=\phi_{\omega_1}\circ\phi_{\omega_2}\cdots\phi_{\omega_n}$ for $\omega=(\omega_1,\cdots\omega_n)\in\N^n$, $n\geq 1$.
   \item $\norm{\phi'_{\omega}}:=\sup_{x\in X}\abs{\phi'_{\omega}(x)}$ for $\omega\in\bigcup_{n\geq 1}\N^n $.
   \item $C(X)$ denotes the space of continuous functions on $X$.
   \item $\sigma:\N^{\infty} \to \N^{\infty}$ is the shift tramsformation.
   \item $\norm{\cdot}_{\infty}$ denotes the supremum norm on $C(X)$.
   
 \end{itemize}
 \medskip

 Let $\partial X$ denote the boundary of $X$ and Int$X$ denote the interior of $X$. Our definitions agree with the classical one for conformal iterated function system. For details, see \cite{book1} and \cite{A.H.Fan}.
 
\medskip
An iterated function system $S=\{\phi_i\}_{i\in \N}$ (IFS for short)  is said to be conformal if the following conditions are satisfied:
\begin{enumerate}
  \item[(1)] The open set condition is satisfied for $\text{Int}(X)$, i.e., $\phi_i(\text{Int}X)\subset\text{Int}(X)$ for each $i\in\N$ and 

  $\phi_i(\text{Int}X)\bigcap \phi_j(\text{Int}X)=\emptyset$ for each $(i,j)\in\N^2,i\neq j$.
  \item[(2)] There exists an open connected set $V$ with $X\subset V\subset \R^d$ such that all maps $\phi_i$, $i\in\N$, extend to $C^1 $ conformal  diffemorphisms of $V$ into $V$.
  \item[(3)] There exist $h,l>0$ such that for each $x\in\partial X\subset \R^d$, there exists an open cone Con$(x,h,l)\subset$ Int$X$ with vertex $x$, central angle of Lebesgue measure $h$ and altitude $l$.   
  \item[(4)]  There exists $K\geq 1$ such that $\abs{\phi_{\omega}'(y)}\leq K\abs{\phi_{\omega}'(x)}$ for every $\omega \in \cup_{n\geq 1} \N^n$ and every pair of points $x,y\in V$ (Bounded distortion property). 
\end{enumerate}
\medskip

 The topological pressure function for the  conformal iterated function system $S=\{\phi_i: X\to X:i\in\N\}$ is defined as
\begin{equation}
  P(t)= \lim_{n\to\infty}\frac{1}{n}\log \sum_{\omega\in\N^n}\norm{\phi'_{\omega}}^t.
\end{equation}
The system $S$ is said to be regular if there exists $t\geq 0$ such that $P(t)=0.$

\medskip
 Let $X=[0,1]$. The P2GLE dynamical system is associated with an IFS:
$$S=\{\phi_i(x)=2^{-i}(x+1):i\in\N\}.$$ 
In fact, the functions $\phi_i$ are inverse branches of the transformation $T$ defined in \eqref{expansion}.
The coding map $\pi:N^{\infty}\to X$ is defined by 
$$\pi(\omega)=\bigcap_{n=1}^{\infty}\phi_{\omega|_n}(X)=\sum_{n=1}^{\infty}2^{-(\omega_1+\omega_2+\cdots\omega_n)}\quad\text{for any }\omega=(\omega_1,\omega_2\cdots). $$

\medskip
Notice that for each $i\in\N$ and for all $x,y\in[0,1]$,
$$\abs{\phi_i(x)-\phi_i(y)}\leq \frac{1}{2}\abs{x-y}.$$

\begin{lemma}
  The IFS $S$ associated to P2GLE is conformal.
\end{lemma}
\begin{proof}(1) The open set condition is satisfied for $\text{Int}X=(0,1)$, since $$\phi_i(\text{Int}X)=(2^{-i},2^{-i+1})\subset\text{Int}X \quad\text{for each } i\in\N,$$  and $$\phi_i(\text{Int}X)\bigcap \phi_j(\text{Int}X)=\emptyset \quad\text{ for each } (i,j)\in\N^2,i\neq j.$$\\
(2) Let $V=(-1,2)$. It follows that $X\subset V$ and  all maps $\phi_i, i\in\N$, extend to $C^1 $ conformal  diffemorphisms of $V$ into $V$.\\
(3) The third condition of conformal IFS was proved by \cite[Section 2]{ifs}. \\
(4) The bounded distortion property holds, since $\abs{\phi'_{\omega}(y)}=\abs{\phi'_{\omega}(x)}=2^{-(\omega_1+\omega_2+\cdots + \omega_n)}$ for every $\omega \in \N^n, n\geq 1$ and every pair of points $x,y\in V.$
\end{proof}

\begin{lemma}
  The iterated function system $S$ is regular.
\end{lemma}
\begin{proof}
  Observe that $$\norm{\phi'_{\omega}}=\sup_{x\in X}\abs{\phi'_{\omega}(x)}=2^{-(\omega_1+\omega_2+\cdots +\omega_n)}.$$
  By definition, the topological pressure function for $S$ is
  \begin{align*}
    P(t)={} & \lim_{n\to\infty}\frac{1}{n}\log \sum_{\omega\in\N^n}\norm{\phi'_{\omega}}^t\\
    ={} & \lim_{n\to\infty}\frac{1}{n}\log \sum_{\omega_1\in\N}\sum_{\omega_2\in\N}\cdots\sum_{\omega_n\in\N}2^{-t(\omega_1+\omega_2+\cdots +\omega_n)}\\
    ={} & -\log(2^{t}-1).
  \end{align*}
It is easily seen that $P(1)=0$. Hence, $S$ is regular. 
\end{proof}

\subsection{Ruelle  operator theory}
Let $\beta>0$. A H{\"o}lder family of functions of order $\beta$ is a family of continuous functions $$G=\{g^{(i)}:X\to \R\}_{i\in\N}$$ such that 
 $$V_{\beta}(G)=\sup_{n\geq 1}V_n(G)<\infty,$$where
 $$V_n(G)=\sup_{\omega\in\N^n}\sup_{x,y\in X}\Big\{\left|{g^{(\omega_1)}(\phi_{\sigma(\omega)}(x))-g^{(\omega_1)}(\phi_{\sigma(\omega)}(y))}\right|\Big\}e^{\beta(n-1)}.$$

 \medskip
 A H\"{o}lder family of functions $G=\{g^{(i)}:X\to \R\}_{i\in\N}$ is said to be strong if 
 \begin{equation}\label{strong}
   \sum_{i\in\N}\norm{e^{g^{(i)}}}_{\infty}<\infty.
 \end{equation}

 \medskip
 The Ruelle operator on $C(X)$ associated with $G$ is defined by
$$\mathcal{L}_{G}(h)(x):=\sum_{i\in\N}e^{g^{(i)}}h(\phi_i(x)),\quad\forall  h\in C(X).$$
Denote by $\mathcal{L}^*_{G}$ the dual of $\mathcal{L}_{G}$.

\medskip
 The topological pressure of $G$ is defined by
 \begin{equation}\label{pressure}
   P(G):=\lim_{n\to\infty}\frac{1}{n}\log\sum_{\abs{\omega}=n}\exp(\sup_{x\in X}\sum_{j=1}^ng^{(\omega_j)}\circ \phi_{\sigma^j\omega}(x)).
 \end{equation}

 \medskip
Let $S=\{\phi_i:X\to X\}_{i\in\N}$ be a conformal iterated function system, $\Psi=\{\psi_i: X\to\R, i\in\N\}$, and $F=\{f_i: X\to\R, i\in\N\}$ be two families of real-valued H{\"o}lder functions. The amalgamated functions on $\N^{\infty}$ associated with $\Psi$ and $F$ are defined as 
$$\tilde{\psi}(\omega)=\psiassert,\qquad \tilde{f}(\omega)=\fassert.$$



\medskip
The following theorems will be useful.
\begin{theorem}({\cite[Theorem 6.4]{16}})\label{core}
  Let $\Psi=\{\psi_i\}_{i\in\N}$ and $F=\{f_i\}_{i\in\N}$ be two H\"{o}lder families of functions. Suppose the sets $\{i\in\N:\sup_{x}(\psi_i(x))>0\}$ and $\{i\in\N:\sup_{x}(f_i(x))>0\}$ are finite. Then the function $(t,q) \mapsto P(t,q)=P(t\Psi-qF)$ is real analytic with respect to $(t,q)\in \text{Int}(D)$, where 
  $$D=\Big\{(t,q):\sum_{i\in\N}\exp(\sup_x(t\psi_i(x)-qf_i(x)))<\infty\Big\}.$$
\end{theorem}

\begin{theorem}({\cite[Section 2]{16}})\label{measure}
  Suppose that $t\Psi-qF$ is a strong H\"{o}lder family for any $(t,q)\in D$. For each $(t,q)\in D$, there exists a unique $(t\Psi-qF)$-conformal probability measure $\nu_{t,q}$ on $X$ such that $\mathcal{L}^*_{t\Psi-qF}\nu_{t,q}=e^{P(t,q)}\nu_{t,q}$, and a unique shift invariant ergodic probability measure $\tilde{\mu}_{t,q}$ on $\N^{\infty}$ such that the measure $\mu_{t,q}:=\tilde{\mu}_{t,q}\circ \pi^{-1}$ on $X$ is equivalent to $\nu_{t,q}$ and the Gibbs property is satisfied:
  $$\frac{1}{C}\leq \dfrac{\tilde{\mu}_{t,q}([\omega|_n])}{\exp(\sum_{j=1}^n(t\Psi-qF)^{(\omega_j)}(\pi(\sigma^j\omega))-nP(t,q))}\leq C\quad\text{for all } \omega\in\N^{\infty}.$$
\end{theorem}

\begin{theorem}({\cite{16}})\label{thm.6}
If $t\Psi-qF$ is a strong  H{\"o}lder family for any $(t,q)\in D$ and 
  $$\int\Big(\abs{\tilde{\psi}}+\abs{\tilde{f}}\Big) \,\mathrm{d}\tilde{\mu}_{t,q}<\infty,$$  then 
  \begin{equation}\label{ergodic-reason}
    \dfrac{\partial P}{\partial t}=\int \tilde{\psi} \,\mathrm{d} \tilde{\mu}_{t,q} \quad\text{ and }\quad\dfrac{\partial P}{\partial q}=-\int \tilde{f} \,\mathrm{d} \tilde{\mu}_{t,q}.
  \end{equation}
  If $t\tilde{\psi}-q\tilde{f}$ is not cohomologous to a constant with respect to the transformation $T$, i.e., there is no function $u\in C(X)$ such that 
  \[t\tilde{\psi}-q\tilde{f}=u-u\circ T+C,\]
  where $C$ is a constant, then $P(t,q)$ is strictly convex and the matrix
  \[H(t,q):=
    \begin{bmatrix}
      \dfrac{\partial^2 p}{\partial t^2} & \dfrac{\partial^2 p}{\partial t\partial q}\\[8pt]
      \dfrac{\partial^2 p}{\partial t\partial q} & \dfrac{\partial^2 p}{\partial q^2}
    \end{bmatrix}
      \]is positive definite.

\end{theorem}
\medskip
    Consider the level sets $$E_{\varphi}(\xi)=\{x\in(0,1]:\bar{\varphi}(x)=\xi\},\quad\xi\in\R.$$ Let $P(t,q)$ be the corresponding pressure function of the potential $t\Psi-qF$ and $D$ be the analytic area of $P(t,q)$. By \cite[Lemma 4.7]{A.H.Fan}, we know that if $\tilde{\psi}(\omega)=-\log\abs{T'(\pi(\omega))}$ and $\tilde{f}(\omega)=-\varphi(\pi(\omega))$. Then
    \begin{equation}
    -\int \log\abs{T'(x)}\diff \mu_{t,q}=\int  \tilde{\psi}\diff \tilde{\mu}_{t,q} \quad\text{ and }\quad \int \varphi(x)\diff \mu_{t,q}=-\int  \tilde{f}\diff \tilde{\mu}_{t,q}, 
    \end{equation}
    and $t\tilde{\psi}-q\tilde{f}$ is not cohomologous to a constant.
    \medskip

 Applying Theorem \ref{thm.6}, we arrive at the following proposition.
\begin{proposition}\label{prop.1}({\cite[Proposition 4.8]{A.H.Fan}})
  The following statements hold on $D$.
\begin{item}
(1) $P(t,q)$ is analytic and strictly convex.
\end{item}

\begin{item}
(2) $P(t,q)$ is strictly decreasing and strictly convex with respect to $t$, i.e., $(\partial p/\partial t) (t,q)<0$ and $(\partial^2 p/\partial t^2) (t,q) >0.$ Furthermore,
\begin{equation}
  \frac{\partial p}{\partial t}(t,q)=-\int \log \abs{T^{\prime}(x)} \diff \mu_{t,q}.
\end{equation}
\end{item}

\begin{item}
  (3) $P(t,q)$ is strictly increasing and strictly convex with respect to $q$, i.e., $(\partial p/\partial q) (t,q)>0$ and $(\partial^2 p/\partial q^2) (t,q) >0.$ Furthermore,
  \begin{equation}\label{equa.1}
    \frac{\partial p}{\partial q}(t,q)=\int  \varphi(x)\diff \mu_{t,q}.
  \end{equation}
  \end{item}

  \begin{item}
  (4) The matrix \[H(t,q):=
  \begin{bmatrix}
    \dfrac{\partial^2 p}{\partial t^2} & \dfrac{\partial^2 p}{\partial t\partial q}\\[8pt]
    \dfrac{\partial^2 p}{\partial t\partial q} & \dfrac{\partial^2 p}{\partial q^2}
  \end{bmatrix}
    \]is positive definite.
  \end{item}
 \end{proposition}

 \section{Khintchine spectrum}
 Recall that the transformation $T$ on $X=(0,1]$ associated to P2GLE is defined as
 \begin{equation*}
   Tx=2^nx-1 \quad\text{for } x\in(1/2^n,1/2^{n-1}] ,\quad n\in\N
 \end{equation*} and the corresponding conformal iterated function system on $X$ is
 $$S=\Big\{\phi_{i}(x)=2^{-i}(x+1):i\in\N\Big\}.$$ 
 Now we will calculate the Hausdorff dimensions of the level sets $$E(\xi)=\Big\{x\in(0,1]:\lim_{n\to\infty}\frac{1}{n}\sum_{j=1}^nd_j(x)=\xi\Big\}=\Big\{x\in(0,1]:\lim_{n\to\infty}\frac{1}{n}\sum_{j=0}^{n-1}\lceil-\log_2(T^{j}(x))\rceil=\xi\Big\},\quad \xi\geq 1.$$

 \medskip
 Let $\Psi=\{\log\abs{\phi'_i}:i\in\N\}=\{-i\log2:i\in\N\}$ and $F=\{-i:i\in\N\}$. It can be checked that $\Psi$ and $F$ are two families of real-valued H{\"o}lder functions. 
 \begin{proposition} Let $D:=\{(t,q):q-t\log2<0\}$. For any $(t,q)\in D,$ we have
   \begin{itemize}
     \item[(1)] the family $t\Psi-qF$ is H{\"o}lder and strong; 
     \item[(2)] the topological pressure $P$ of the potential $t\Psi-qF$ is
     \begin{equation}
       P(t,q)=\log\dfrac{e^{-t\log2+q}}{1-e^{-t\log2+q}}.
     \end{equation}
   \end{itemize}
   
 \end{proposition}
 
 \begin{proof}
  (1) We first show that the family $t\Psi-qF$ is H{\"o}lder. Since $$(t\Psi-qF)^{(i)}:=t\log\abs{\phi'_i}+qi=i(-t\log2+q),$$
   we have $V_n(t\Psi-qF)=0.$ By definition, $t\Psi-qF$ is H{\"o}lder. 
 
   For $(t,q)\in D$, replacing $g^{(i)}$ by $i(-t\log2+q)$ in \eqref{strong}, we obtain that 
   $$\sum_{i\in\N}\norm{e^{i(-t\log2+q)}}_{\infty}<\infty.$$
   Thus $t\Psi-qF$ is strong.
 
   \medskip
   (2) It suffices to notice that $(t\Psi-qF)^{(i)}$ is a constant function for each $i\in\N$. Thus
   $$\sup_{x\in X}\sum_{j=1}^n(t\Psi-qF)^{(\omega_j)}\circ \phi_{\sigma^j\omega}(x)=(-t\log2+q)\sum_{j=1}^n\omega_j.$$
  By definition,
   \begin{align*}
     P(t,q)={}& \lim_{n\to\infty}\frac{1}{n}\log\sum_{\omega_1,\omega_2\cdots\omega_n\in\N}\exp((-t\log2+q)\sum_{j=1}^n\omega_j)\\
     ={}& \lim_{n\to\infty}\frac{1}{n}\log(\frac{e^{-t\log2+q}}{1-e^{-t\log2+q}})^n\\
     ={}& \log\dfrac{e^{-t\log2+q}}{1-e^{-t\log2+q}}.
   \end{align*}
 \end{proof}
 \subsection{Proof of \eqref{spec.1} in Theorem \ref{thm.1}  }
 Let us find the solution of the system \eqref{equa.pressure}.
 \begin{proposition}\label{q}
   For any $\xi\in(1,\infty)$, the unique solution of the  system 
   \begin{equation}
  \begin{cases}
 P(t,q)=q\xi,\\
 \dfrac{\partial P}{\partial q}(t,q)=\xi
   \end{cases}
  \end{equation}
   is  
   \[\begin{cases}\label{equa.3.1}
     t(\xi)=\dfrac{\frac{\log(\xi-1)}{\xi}+\log\xi-\log(\xi-1)}{\log2},\\
     q(\xi)=\dfrac{\log(\xi-1)}{\xi}.
   \end{cases}
     \]
 \end{proposition}
 
 \begin{proof}
   Since \begin{equation}\label{P(t,q)}
     P(t,q)=\log\dfrac{e^{-t\log2+q}}{1-e^{-t\log2+q}},
   \end{equation}
     we have
   \begin{equation}\label{partial}
     \frac{\partial P}{\partial q}(t,q)=\dfrac{1}{1-e^{-t\log2+q}}.
   \end{equation}
   Substituting \eqref{P(t,q)} and \eqref{partial} into the above system yields
  $$ e^{-t\log2+q}=\dfrac{e^{q\xi}}{1+e^{q\xi}}=\dfrac{\xi-1}{\xi}.$$
  Hence, $q(\xi)=\dfrac{\log(\xi-1)}{\xi}$ and $t(\xi)=\dfrac{\frac{\log(\xi-1)}{\xi}+\log\xi-\log(\xi-1)}{\log2} .$
   
 \end{proof}

\begin{lemma}
  For any $\xi>1$, let $\big(t(\xi),q(\xi)\big)$ be the solution of \eqref{equa.3.1}. Then the measure  $\mu_{t(\xi),q(\xi)}$ obtained by Theorem \ref{measure} is supported on $E(\xi)$, i.e., $\mu_{t(\xi),q(\xi)}(E(\xi))>0.$
\end{lemma}

\begin{proof}
  By \eqref{equa.1}, 
  $$\int \lceil-\log_2 x\rceil\,\mathrm{d}\mu_{t,q}= \dfrac{\partial
   P}{\partial q}(t,q).$$
  Take $(t,q)=\big(t(\xi),q(\xi)\big)$. By the ergodicity of $\mu_{t(\xi),q(\xi)}$ and the second equation in \eqref{equa.3.1}, we have that for $\mu_{t(\xi),q(\xi)}$-almost every $x$,
  \begin{equation*}
    \lim_{n\to\infty}\frac{1}{n}\sum_{j=0}^{n-1}\lceil-\log_2(T^{j}(x))\rceil=\int \lceil-\log_2 x\rceil\,\mathrm{d}\mu_{t(\xi),q(\xi)}=\dfrac{\partial P}{\partial q}\big(t(\xi),q(\xi)\big)=\xi.
  \end{equation*}
  Therefore, $\mu_{t(\xi),q(\xi)}(E(\xi))=1>0.$

\end{proof}
\medskip
The following lemma is a consequence of Theorem \ref{measure}. We use the notation $f\asymp g$ to mean that $\frac{1}{C}f(x)\leq g(x)\leq Cf(x)$ for some constant $C>0$ and for all $x\in X.$
\begin{lemma}
  For $x\in X$, let $I_n(x):=\{y\in X:d_i(y)=d_i(x), 1\leq i\leq n\}$ and $\omega=\pi^{-1}(x)$. Then 
  \begin{equation}\label{lemma}
    \mu_{t,q}(I_n(x))\asymp \exp(\sum_{j=1}^n(t\Psi-qF)^{(\omega_j)}(\pi(\sigma^j\omega))-nP(t,q)).
  \end{equation}
    
  \begin{proof}
    Notice that $\mu_{t,q}(I_n(x))=\tilde{\mu}_{t,q}([\omega|_n])$. Then $\mu_{t,q}$ is a Gibbs measure on $X$, i.e., there exists $C>0$, such that
  $$\frac{1}{C}\leq \dfrac{\mu_{t,q}(I_n(x))}{\exp(\sum_{j=1}^n(t\Psi-qF)^{(\omega_j)}(\pi(\sigma^j\omega))-nP(t,q))}\leq C.$$
  \end{proof}

\end{lemma}

\begin{theorem}\label{them.dim}
      For any $\xi> 1$, we have 

 \begin{equation}\label{equa.Khintchine}
 \dim_HE(\xi)=t(\xi)=\dfrac{\frac{\log(\xi-1)}{\xi}+\log\xi-\log(\xi-1)}{\log2}. 
\end{equation}
\end{theorem}
\medskip
\begin{proof} 
  
  Let $(t,q)\in D$. Since $$(t\Psi-qF)^{(i)}=i(-t\log2+q),$$
we have
  $$\exp(\sum_{j=1}^n(t\Psi-qF)^{(\omega_j)}(\pi(\sigma^j\omega))-nP(t,q))=\exp((-t\log2+q)\sum_{j=1}^n\omega_j-nP(t,q)).$$
By Lemma \ref{lemma}, 
\begin{equation}\label{equa.point}
\mu_{t,q}(I_n(x))\asymp\exp((-t\log2+q)\sum_{j=1}^nd_j-nP(t,q)).
\end{equation}

\medskip
  Fix $\xi>1$. For any $x\in E(\xi)$, we have
  $$\lim_{n\to\infty}\frac{1}{n}\sum_{j=1}^nd_j(x)=\xi.$$
  Observe that $$\abs{I_n(x)}=2^{-d_1(x)+d_2(x)+\cdots d_n(x)}.$$
  Take $(t,q)=\big(t(\xi),q(\xi)\big)$. Then $P(t,q)=q\xi$, and
   \begin{align*}
    d_{\mu}(x)={} & \liminf_{n \to \infty}\dfrac{\log\mu_{t,q}(I_n(x))}{\log\abs{I_n(x)}}\\
    ={} &\liminf_{n \to \infty}\dfrac{(-t\log2+q)\sum_{j=1}^nd_j-nP(t,q)}{-(\sum_{j=1}^nd_j)\log2}\\
    ={}&\dfrac{-t\xi\log2+q\xi-q\xi}{-\xi\log2}\\
    ={}& t.
   \end{align*}
   By Billingsley's lemma (see \cite[Lemma 1.4.1]{Billingsley}), for $\xi>1$, we have $\dim_HE(\xi)=t(\xi)$.
\end{proof}
\medskip
  Summarizing, we have proved that the unique solution $t(\xi)$ of \eqref{equa.3.1} is the Khintchine spectrum for $\xi>1$. Now let us discuss the properties of the function $\xi\mapsto t(\xi)$.
\subsection{Proofs of Theorem \ref{thm.1} (1)-(4)}
Recall that the level set 
 \[E(2)=\{x\in(0,1]:\gamma(x)=2\}\] has Lebesgue measure one. This implies 
 $\dim_HE(2)=1.$ 
 \medskip

 By Theorem \ref{them.dim}, 
 \begin{equation*}\label{equa.3.2.1}
\dim_HE(\xi)=t(\xi)=\dfrac{\frac{\log(\xi-1)}{\xi}+\log\xi-\log(\xi-1)}{\log2},\quad\text{for all }\xi>1.
 \end{equation*}
Letting $\xi\to1$ and $\xi\to\infty$ respectively, we have
$$\lim_{\xi\to 1} t(\xi)=0,\qquad
\lim_{\xi\to \infty} t(\xi)=0.$$
This is assertion (2) of Theorem \ref{thm.1}.
\medskip

By differentiating the equation in \eqref{equa.Khintchine} with respect to $\xi$, we arrive at the following proposition.
\begin{proposition}\label{prop.derivative}
  For $\xi>1$,
  \begin{equation}\label{equa.see}
    t^{\prime}(\xi)=\dfrac{\log(\xi-1)}{-\xi^2\log2}.
  \end{equation}
\end{proposition}
\medskip
  By \eqref{equa.see}, we see immediately that $t^{\prime}(\xi)>0$ for $\xi<2$, $t'(2)=0$, and $t'(\xi)<0$ for $\xi>2$. Furthermore,
  \begin{equation*}
    \lim_{\xi\to 1} t'(\xi)=\infty,\qquad
    \lim_{\xi\to \infty} t'(\xi)=0.
  \end{equation*}
We are now in a position to show the boundary case of Theorem \ref{thm.1}.
\begin{proposition}\label{thm.5}
  For $\xi=1$, $\dim_H E(1)=0$. 
\end{proposition}
\begin{proof}
It is sufficient to show that \begin{equation}\label{equa.12.3}\dim_H E(1)\leq t\log2,\qquad \text{for all }t>0.\end{equation}
   For any $t>0$, from $\lim_{\xi\to1}t(\xi)=0$ and $q(\xi)=\dfrac{\log(\xi-1)}{\xi}$, it follows that there exists $\xi>1$ such that $0<t(\xi)<t$ and $q(\xi)<0$. Since $P(t,q)$ is strictly decreasing with respect to $t$, we can choose $\varepsilon_0>0$ such that
   $$\frac{P(t,q(\xi))-P\big(t(\xi),q(\xi)\big)}{q(\xi)}>\varepsilon_0.$$
   \medskip

   For $n\geq 1$, set
   $$E_0^n(\varepsilon_0):=\Big\{x\in[0,1):\frac{1}{n}\sum_{j=1}^n d_j(x)<\frac{1}{\log2}(\xi+\varepsilon_0)\Big\}.$$
  Since $\frac{1}{\log2}(\xi+\varepsilon_0)>1$, $$E(1)\subset \bigcup_{N=1}^{\infty}\bigcap_{n=N}^{\infty}E_0^n(\varepsilon_0).$$
  Let $\mathcal{I}(n,\xi,\varepsilon_0)$ denote the family of all $n$th-order cylinders $I_n(d_1,\cdots, d_n)$ such that 
  $$\frac{1}{n}\sum_{j=1}^n d_j(x)<\frac{1}{\log2}(\xi+\varepsilon_0).$$
  Then $$E_0^n(\varepsilon_0)=\bigcup_{J\in\mathcal{I}(n,\xi,\varepsilon_0)}J.$$
  Hence, $\{J:J\in\mathcal{I}(n,\xi,\varepsilon_0), n\geq1\}$ is a cover of $E(1)$. It remains to prove that 
  $$\sum_{n=1}^{\infty}\sum_{J\in\mathcal{I}(n,\xi,\varepsilon_0)}\abs{J}^{t\log2}<\infty.$$
  By  \eqref{equa.point} and the fact that $\abs{I_n(x)}=2^{-(d_1(x)+\cdots +d_n(x))}$, we have
  $$\mu_{t,q}(I_n(x))\asymp\exp(-nP(t,q))2^{q\sum_{j=1}^nd_j}\abs{I_n(x)}^{t\log2}.$$
  Thus,
\begin{align}
\sum_{n=1}^{\infty}\sum_{J\in\mathcal{I}(n,\xi,\varepsilon_0)}\abs{J}^{t\log2} & =\sum_{n=1}^{\infty}\sum_{\sum_{j=1}^nd_j<(\xi+\varepsilon_0)n\log_2 e}\frac{e^{nP(t,q(\xi))}}{2^{q(\xi)\sum_{j=1}^{\infty}d_j}}\frac{\abs{J}^{t\log2}2^{q(\xi)\sum_{j=1}^{\infty}d_j}}{e^{nP(t,q(\xi))}}\\
& \leq C \sum_{n=1}^{\infty}e^{nP(t,q(\xi))-(\xi+\varepsilon_0)q(\xi)}\cdot \sum_{J\in\mathcal{I}(n,\xi,\varepsilon_0)}\mu_{t,q(\xi)}(J)<\infty,
\end{align}
where $C$ is a constant. Hence, we obtain \eqref{equa.12.3}.
\end{proof}

\subsection{Proof of  Theorem \ref{thm.1} (5) }
The following proposition proves the last assertion of Theorem \ref{thm.1}, i.e., $t(\xi)$ has only one inflection point.
\begin{proposition}
  We have
  \begin{equation}
    t^{\prime\prime}(\xi)=\frac{2(\xi-1)\log(\xi-1)-\xi}{\xi^3(\xi-1)\log2},\quad \forall \xi>1.
  \end{equation}
  Further, there exists a point $\tilde{\xi}$ such that $t^{\prime\prime}(\tilde{\xi})=0$ and $t^{\prime\prime}(\xi) < 0$ for $\xi < \tilde{\xi}$, $t^{\prime\prime}(\xi) > 0$ for $\xi > \tilde{\xi}$. 
\end{proposition}

\begin{proof} 
  For $\xi>1$, we have $\xi^3(\xi-1)\log2 > 0$. Define a function $f$ by $f(\xi)=2(\xi-1)\log(\xi-1)-\xi$ for all $\xi>1$. Notice that $$\lim_{\xi\to 1} f(\xi)=-1<0, \qquad \lim_{\xi\to\infty}f(\xi)=\infty.$$ Taking the derivative of $f$ gives
  $f^{\prime}(\xi)=2\log(\xi-1)+1.$
  We then deduce that $f$ is strictly decreasing on $(1,1+e^{-1/2})$ and increasing on $(1+e^{-1/2},\infty)$. Observe that 
$f(3)<0,$ and $f(1+e)>0$. Therefore, there exists exactly one point $\tilde{\xi}\in(3,1+e)$ such that $t^{\prime\prime}(\tilde{\xi})=0.$ The rest of the statements are evident.
  
\end{proof}

\section{Birkhoff spectrum for $\varphi: x\mapsto\log(d_1(x))$}
Let $\varphi(x)=\log(d_1(x))=\log(\lceil-\log_2 x\rceil)$. We consider the Hausdorff dimensions of the sets
$$E_{\varphi}(\xi)=\Big\{x\in(0,1]:\lim_{n\to\infty}\frac{1}{n}\sum_{j=0}^{n-1}\log(d_1(T^jx))=\xi\Big\},\quad \xi\geq 0.$$  Recall that, for Lebesgue almost all \x,
$$\lim_{n\to\infty}\frac{1}{n}\sum_{j=0}^{n-1}\log(d_1(T^jx))=\xi_0=\int_0^1\log\lceil-\log_2 x\rceil \,\mathrm{d}x.$$
Thus $\dim_HE_{\varphi}(\xi_0)=1$.

\medskip
Let $\Psi=\{\log\abs{\phi'_i}:i\in\N\}=\{-i\log2:i\in\N\}$ and $F=\{-\log i:i\in\N\}$. 
\begin{proposition}\label{press} Let $D:=\{(t,q):t>0,\text{ or } t=0, q<-1\}$. For any $(t,q)\in D:$
  \begin{item}
  (1) the family $t\Psi-qF$ is H{\"o}lder and strong; 
  \end{item}

  \begin{item}
    (2) the topological pressure $P(t,q)$ of the potential $t\Psi-qF$ is given by
    \begin{equation}\label{equa.5.2}
      P(t,q)= \log(\sum_{n=1}^{\infty}2^{-n t}\cdot n^q).
    \end{equation}
  \end{item}
\end{proposition}
\begin{proof}
  (1) Since $(t\Psi-qF)^{(i)}=-ti\log2+q\log i,$ we have
   $V_n(t\Psi-qF)=0.$ By definition, the family $t\Psi-qF$ is H{\"o}lder. 
 
   \medskip
   For $(t,q)\in D$, it can be checked that 
   $$\sum_{i\in\N}e^{-ti\log2+q\log i}<\infty.$$
   Thus $t\Psi-qF$ is strong.
   \medskip

   (2) Observe that
   $$\sup_{x\in X}\sum_{j=1}^n(t\Psi-qF)^{(\omega_j)}\circ \phi_{\sigma^j\omega}(x)=-t\log2\sum_{j=1}^n\omega_j+q\sum_{j=1}^n\log\omega_j.$$
  By definition,
   \begin{align*}
     P(t,q)
     ={}& \lim_{n\to\infty}\frac{1}{n}\log\sum_{\omega_1,\omega_2\cdots\omega_n\in\N}\exp(-t(\log2)\sum_{j=1}^n\omega_j+q\sum_{j=1}^n\log\omega_j)\\
     ={}& \lim_{n\to\infty}\frac{1}{n}\log(\sum_{\omega=1}^{\infty}\exp(-t\omega\log2+q\log\omega))^n\\
     ={}& \log(\sum_{n=1}^{\infty}2^{-nt}\cdot n^q).
   \end{align*}
 \end{proof}

By Theorem \ref{core} and Proposition \ref{press}, we know that $D$ is the analytic area of the pressure $P(t,q).$ From Mayer \cite{Mayer} (see also Pollicott and Weiss \cite[Proposition 4]{pollicott}), we know that $\mu_{1,0}$ is the Lebesgue measure on $[0,1]$. By \eqref{equa.1}, we obtain

\begin{equation}
  \frac{\partial P}{\partial q}(1,0)=\int \log\lceil-\log_2 x\rceil \,\mathrm{d}x=\xi_0.
\end{equation}

\subsection{Further sutdy on $P(t,q)$}

 Now we show some properties of $P(t,q)$ and $\partial P/\partial q$.
  \begin{proposition}\label{prop.0} For fixed $t>0$, 
    \begin{equation}\label{equa.5.8}
      \frac{\partial P}{\partial q}(t,q)=\frac{\sum_{n=1}^{\infty}2^{-nt}(\log n)n^q}{\sum_{n=1}^{\infty}2^{-nt}n^q},\quad \forall q\in\R.
    \end{equation}
    Moreover, we have
    \begin{equation}\label{equa.5.5}
      P(t,0)=-\log(2^t-1),
    \end{equation}
    and 
    \begin{equation}\label{equa.5.6}
      P(0,q)=\log\zeta(-q),
    \end{equation}
    where $\zeta$
    is the Riemann zeta function, defined by 
    $$\zeta(s):=\sum_{n=1}^{\infty}\frac{1}{n^s},\quad \forall s>1.$$
    Consequently:
    \begin{item}
     (1) $P(1,0)=0,$ and 
     $$\lim_{q\to-\infty}P(0,q)=0;$$
    \end{item}

    \begin{item}
      (2) for fixed $ t>0$, $$\lim_{q\to \infty}P(t,q)=\infty;$$
    \end{item}

    \begin{item}
     (3) for fixed $t\geq0$, 
      \begin{equation}\label{equa.5.4}
        \lim_{q\to-\infty}\frac{\partial P}{\partial q}(t,q)=0;
      \end{equation}
    \end{item}
    \begin{item}
      (4) we have 
      \begin{equation}\label{equa.5.7}
        \lim_{t\to0}\frac{\partial P}{\partial q}(t,0)=\infty.
      \end{equation}
    \end{item}
  \end{proposition}
  \begin{proof}
    Fix $t>0$ and $q_0>0$. Notice that for any $n\in\N $ and $q\leq q_0$, 
    $$2^{-nt}(\log n)n^q\leq2^{-nt}(\log n)n^{q_0}.$$  Since the series $\sum_{n=1}^{\infty}2^{-nt}(\log n)n^{q_0}$ is convergent, we arrive at \eqref{equa.5.8}.
   Since $$P(t,0)=\log(\sum_{n=1}^{\infty}2^{-n t})=-\log(2^t-1),$$ we have $P(1,0)=0$.
    Notice that
     $$P(0,q)=\log(\sum_{n=1}^{\infty}n^q)=\log\zeta(-q),$$ 
and $$\lim_{q\to-\infty}\zeta(-q)=1.$$ 
Then, we have $P(0,q)\to0$  $(q\to-\infty)$. We have thus proved \eqref{equa.5.5}, \eqref{equa.5.6} and part (1).
\medskip

Next, let $q\to\infty$. We note that $n^q\to\infty$ for all $n>1$, hence $P(t,q)\to\infty.$ This proves part (2).


\medskip

Now, let $q\to-\infty$. Since $n^q\to0$, we have 
\[\lim_{q\to-\infty}\frac{\partial P}{\partial q}(t,q)=\lim_{q\to-\infty}\frac{\sum_{n=1}^{\infty}2^{-nt}(\log n)n^q}{2^{-t}+\sum_{n=2}^{\infty}2^{-nt}n^q}=0,\]
which is the assertion of part (3).

\medskip
For the final part, let $q=0$.
By \eqref{equa.5.8}, we have
\begin{equation*}
  \frac{\partial P}{\partial q}(t,0)=\frac{\sum_{n=1}^{\infty}2^{-nt}(\log n)}{\sum_{n=1}^{\infty}2^{-nt}}.
\end{equation*} 
Fix $M>0$. Then there extsts $N\in\N$, such that $\log n>M$ for any $n\geq N$. Hence, $$\frac{\partial P}{\partial q}(t,0)>\frac{(\log2)\sum_{n=2}^{k-1}2^{-nt}+M\sum_{n=k}^{\infty}2^{-nt}}{\sum_{n=1}^{\infty}2^{-nt}}=2^{-2t}(\log2)(1-2^{-(k-2)t})+M2^{-kt}.
$$
Notice that
\[\lim_{t\to 0}2^{-2t}(\log2)(1-2^{-(k-2)t})+M2^{-kt}=M.
  \]
Thus,
 \[\lim_{t\to0}\frac{\partial P}{\partial q}(t,0)\geq M.\]
 Hence, letting $M\to\infty$, we obtain \eqref{equa.5.7}.
  \end{proof}
  \medskip
  Recall that $\xi_0=(\partial P/\partial q)(1,0)$. Let $\tilde{D}=\{(t,q)\in D:0\leq t\leq 1\}$.
  \begin{proposition}[\cite{A.H.Fan}]\label{4.11}
    For any $\xi\in(0,\infty)$, the system 
    \begin{equation}\label{equa.5.1}
      \begin{cases}
       P(t,q)=q\xi,\\
       \dfrac{\partial P}{\partial q}(t,q)=\xi
         \end{cases}
       \end{equation}
     admits a unique solution $\big(t(\xi),q(\xi)\big)\in\tilde{D}$. For $\xi=\xi_0$, the unique solution is $(t(\xi_0),q(\xi_0))=(1,0)$. Further, the functions $t(\xi)$ and $q(\xi)$ are analytic.
   \end{proposition}
   \medskip
   By differentiating the first equation in system \eqref{equa.5.1} with respect to $\xi$ and applying the second equation in \eqref{equa.5.1}, we obtain the first derivative of function $t$ with respect to $\xi$, i.e.,
  \begin{equation}\label{prop.3}
   t^{\prime}(\xi)=\frac{q(\xi)}{(\partial P/\partial t)\big(t(\xi),q(\xi)\big)},\quad \forall \xi>0.
  \end{equation}
  
\subsection{Proof of Theorem \ref{thm.3} }
  
  As is widely known, if $f$ is  a convex continuously differentiable function on an interval $I$, then $f^{\prime}$ is increasing and 
\[f^{\prime}(x)\leq \frac{f(y)-f(x)}{y-x}\leq f^{\prime}(y)\quad \forall x,y\in I, x<y.
  \] Using this fact, we next present some properties on $t(\xi)$.

  \medskip
The proof of the following proposition is similar to that of \cite[Proposition 4.14]{A.H.Fan}.
  \begin{proposition}\label{prop.2}
    $q(\xi)<0$ for $\xi<\xi_0$; $q(\xi_0)=0$; $q(\xi)>0$ for $\xi>\xi_0$.
  \end{proposition}
  \begin{proof}
    By the convexity of the function $P(1,q)$ with respect to $q$, and the fact $P(1,0)=0$, we have
    $$\frac{P(1,q)-0}{q-0}\geq\frac{\partial P}{\partial q}(1,0)=\xi_0\quad(q>0);\quad\frac{0-P(1,q)}{0-q}\leq\frac{\partial P}{\partial q}(1,0)=\xi_0\quad(q<0).$$
    Thus
    $$P(1,q)\geq \xi_0q ,\quad \forall q\in\R.$$
     Proposition \ref{4.11} shows that $t(\xi)=1$ if and only if $\xi=\xi_0$. Suppose $t\in(0,1)$. For $\xi>\xi_0$, if $q(\xi)\leq 0$, then
     $$P(t,q)> P(1,0)\geq q\xi_0\geq q\xi$$
     which contradicts to the first equation in \eqref{equa.5.1}. Thus 
     $$q(\xi)>0 \quad(\xi>\xi_0).$$
    The rest of the proof runs as before.
     
  \end{proof}

  \medskip
Let us use Proposition \ref{prop.2} and \eqref{prop.3} to deduce some further properties of $t(\xi)$.
  \begin{proposition}\label{proposition.log}
    We have $t^{\prime}(\xi)>0$ for $\xi<\xi_0$, $t^{\prime}(\xi_0)=0$, and $t^{\prime}(\xi)<0$ for $\xi>\xi_0.$ Furthermore,
    \begin{align}
      t(\xi)\to0 &\qquad(\xi\to0),\\
      t(\xi)\to0 &\qquad(\xi\to\infty).\label{equa.5}
    \end{align}

  \end{proposition}
  \begin{proof}
    By Proposition \ref{prop.1}(2), $\partial P/\partial t <0$. Hence, according to Proposition \ref{prop.2} and \eqref{prop.3}, $t(\xi)$ is strictly increasing on $(0,\xi_0)$ and decreasing on $(\xi_0,\infty)$.
Then by the analyticity of $t(\xi)$, we can obtain two analytic inverse functions on the two intervals respectively. For the first inverse function, write $\xi_1=\xi_1(t)$. Then $\xi_1\in(0,\xi_0)$ and $\xi_1^{\prime}(t)>0.$ Considering equations \eqref{equa.5.1} as  equations on $t$, we have
$$\xi_1(t)=\frac{P(t,q(t))}{q(t)}=\frac{\partial P}{\partial q}(t,q(t))
.$$
 Since $\xi_1(t)\in(0,\xi_0)$, we conclude that $q(\xi_1(t))<0$, and in consequence, $P(t,q(\xi_1(t)))<0$. By Proposition \ref{prop.0}(2), $\lim_{q\to\infty}P(t,q)=\infty.$ Thus there exits $q_0(t)$ such that $q_0(t)>q(t)$ and $P(t,q_0(t))=0$. From $\partial^2 p/\partial q^2 >0,$ we see that
\[\xi_1(t)=\frac{\partial P}{\partial q}(t,q(t))<\frac{\partial P}{\partial q}(t,q_0(t)).\]
Since $P(0,q)=\log\zeta(-q)$, then $\lim_{t\to0}q_0(t)=-\infty$. Consequently,
\[\lim_{t\to0}\frac{\partial P}{\partial q}(t,q_0(t))=\lim_{q\to-\infty}\frac{\partial P}{\partial q}(0,q)=0.
  \]
  From the above, it follows that $\lim_{t\to0}\xi_1(t)=0$, and so $\lim_{\xi\to 0}t(\xi)=0.$

  \medskip
  For the second inverse function, write $\xi_2=\xi_2(t)$. Then $\xi_2\in(\xi_0,\infty)$ and $\xi_2^{\prime}(t)<0$. An analysis similar to that in the proof of the first inverse function shows that
  \[\xi_2(t)=\frac{\partial P}{\partial q}(t,q(t))>\frac{\partial P}{\partial q}(t,0).\]
  From \eqref{equa.5.7}, we deduce that $\lim_{\xi\to\infty}t(\xi)=0$.
  \end{proof}
  \medskip
  As in the proof of \cite[Proposition 5.1 ]{A.H.Fan}, we obtain the following proposition.
  \begin{proposition}\label{prop.4}
   We have
    \begin{equation}
  \lim_{\xi\to0}q(\xi)=-\infty.
    \end{equation}
  \end{proposition}


  \begin{proposition}\label{prop.5} We have the following two properties of $\partial P/\partial t$.
    \begin{item}(1) For $t>0$, 
  $$\lim_{q\to-\infty} \frac{\partial P}{\partial t}(t,q)=-\log2.$$
    \end{item}
    \begin{item}
  (2) For $q=0$,
  $$ \lim_{t\to0} \frac{\partial P}{\partial t}(t,0)=-\infty.$$
    \end{item}
  \end{proposition}
  \begin{proof}
    For $t>0$ and $q<-2$, 
    \begin{equation}
      \frac{\partial P}{\partial t}(t,q)=\frac{(-\log2)\sum_{n=1}^{\infty}2^{-n t}n^{q+1}}{\sum_{n=1}^{\infty}2^{-n t}n^{q}}.
    \end{equation}
    Therefore, $$\lim_{q\to-\infty} \frac{\partial P}{\partial t}(t,q)=-\log2.$$
    For $q=0$, $P(t,0)=-\log(2^t-1)$. Therefore, $$\frac{\partial P}{\partial t}(t,0)=\frac{-2^t\log2}{2^t-1}, \qquad \lim_{t\to0} \frac{\partial P}{\partial t}(t,0)=-\infty.$$
  
  \end{proof}
  \medskip
  Combining Propositions \ref{prop.4} and \ref{prop.5} yields
  \[\lim_{\xi\to0}t^{\prime}(\xi)=\infty.
    \]
We have thus proved (1)-(4) of Theorem \ref{thm.3}.

\medskip
Our next goal is to calculate the value of $\dim_H E_{\varphi}(0)$.
\begin{theorem}\label{them.similar}
  For  $\xi=0$, $\dim_H E_{\varphi}(0)=0$.
\end{theorem}
\begin{proof}
  For any $t>0$, since $\lim_{\xi\to 0}t(\xi)=0$, there exists $0<\xi<\xi_0$ such that $0<t(\xi)<t$ and $q(\xi)<0$. Then we can choose $\varepsilon_0>0$, such that 
  $$\frac{P(t,q(\xi))-P\big(t(\xi),q(\xi)\big)}{q(\xi)}>\varepsilon_0.$$
  Let $\mathcal{I}(n,\xi,\varepsilon_0)$ be the family of all $n$th-order cylinders $I_n(d_1,\cdots, d_n)$ such that 
$$\frac{1}{n}\sum_{j=1}^n \log(d_j(x))<\xi+\varepsilon_0.$$
Then $\{J:J\in\mathcal{I}(n,\xi,\varepsilon_0), n\geq1\}$ is a cover of $E_{\varphi}(0)$.  
Since $(t\Psi-qF)^{(i)}=-ti\log2+q\log i$, by Lemma \ref{lemma},
$$\mu_{t,q}(I_n(x))\asymp\exp(-nP(t,q))\abs{I_n(x)}^{t}(d_1\cdots d_n)^q.$$
A calculation similar to that in the proof of Proposition \ref{thm.5} shows that 
\begin{align}
\sum_{n=1}^{\infty}\sum_{J\in\mathcal{I}(n,\xi,\varepsilon_0)}\abs{J}^{t} & =\sum_{n=1}^{\infty}\sum_{(d_1\cdots d_n)<e^{n(\xi+\varepsilon_0)}}\frac{e^{nP(t,q(\xi))}}{(d_1\cdots d_n)^{q(\xi)}}\frac{\abs{J}^{t}(d_1\cdots d_n)^{q(\xi)}}{e^{nP(t,q(\xi))}}\\
& \leq C \sum_{n=1}^{\infty}e^{nP(t,q(\xi))-(\xi+\varepsilon_0)q(\xi)}\cdot \sum_{J\in\mathcal{I}(n,\xi,\varepsilon_0)}\mu_{t,q(\xi)}(J)<\infty.
\end{align}
Thus $\dim_H E_{\varphi}(0)\leq t$. By the arbitrariness of $t$, $\dim_H E_{\varphi}(0)=0$.
\end{proof}
\medskip

For a rigorous proof of the last assertion of Theorem \ref{thm.3}, we refer the reader to \cite[p.103]{A.H.Fan}. A point that should be stressed is that the number of inflection points is not clear yet.

\section{Birkhoff spectrum for $\varphi:x\mapsto2^{d_1(x)}$}
Consider the level sets 
$$E_{\varphi}(\xi)=\Big\{x\in(0,1]:\lim_{n\to\infty}\frac{1}{n}\sum_{j=0}^{n-1}2^{d_1(T^jx)}=\xi\Big\}, \quad \xi\geq 2$$
with $\varphi(x)=2^{d_1(x)}$.
\medskip

Let $\Psi=\{\log\abs{\phi'_i}:i\in\N\}=\{-i\log2:i\in\N\}$ and $F=\{-2^i:i\in\N\}$.
It can be checked that $\Psi$ and $F$ are two families of real-valued H{\"o}lder functions. 
\begin{proposition} Let $D:=\{(t,q):t>0, q=0 \text{ or } q<0\}$. For $(t,q)\in D:$
  \begin{item}
    (1) the family $t\Psi-qF$ is H{\"o}lder and strong; 
  \end{item}
  \begin{item}
   (2) the topological pressure $P$ of the potential $t\Psi-qF$ is
    \begin{equation}
      P(t,q)= \log(\sum_{n=1}^{\infty}2^{-nt}e^{2^{n}q}).
    \end{equation}
  \end{item}
\end{proposition}

    \begin{proof}
 (1) Since $(t\Psi-qF)^{(i)}=-ti\log2+q2^i,$
  $V_n(t\Psi-qF)=0.$ By definition, $t\Psi-qF$ is H{\"o}lder. 
\medskip

  For $(t,q)\in D$, we have 
  $$\sum_{i\in\N}e^{-ti\log2+q2^i}=\sum_{i\in\N}2^{-i t}e^{2^iq}<\infty.$$
  Thus $t\Psi-qF$ is strong.
  \medskip

  (2) Observe that
  $$\sup_{x\in X}\sum_{j=1}^n(t\Psi-qF)^{(\omega_j)}\circ \phi_{\sigma^j\omega}(x)=-t\log2\sum_{j=1}^n\omega_j+q\sum_{j=1}^n2^{\omega_j}.$$
 By definition,
  \begin{align*}
    P(t,q)
    ={}& \lim_{n\to\infty}\frac{1}{n}\log\sum_{\omega_1,\omega_2\cdots\omega_n\in\N}\exp(-t\log2\sum_{j=1}^n\omega_j+q\sum_{j=1}^n2^{\omega_j})\\
    ={}& \lim_{n\to\infty}\frac{1}{n}\log(\sum_{\omega=1}^{\infty}2^{-\omega t}e^{2^{\omega}q})^n\\
    ={}& \log(\sum_{n=1}^{\infty}2^{-nt}e^{2^{n}q}).
  \end{align*}
\end{proof}
\medskip
 Let us investigate some properties of $P(t,q)$.  Let $D_1=\{(t,q):t\geq0,q<0\}$.
 \begin{proposition}\label{prop.6}
  For $(t,q)\in D_1$,
  \begin{equation}\label{equa.6.3}
    \frac{\partial P}{\partial q}(t,q)=\frac{\sum_{n=1}^{\infty}2^{-n(t-1)}e^{2^{n}q}}{\sum_{n=1}^{\infty}2^{-n t}e^{2^{n}q}}.
  \end{equation}
  Consequently,
  \begin{item}
    (1) for fixed $0<t\leq1$,
    \begin{equation}\label{equa.6.4}
      \lim_{q\to0}\frac{\partial P}{\partial q}(t,q)=\infty;
    \end{equation}
  \end{item}
  \begin{item}
    (2) we have
    \begin{equation}\label{equa.6.11}
      \lim_{(t,q)\to(0,0)}\frac{\partial P}{\partial q}(t,q)=\infty;
    \end{equation}
  \end{item}
  \begin{item}
    (3) for fixed $t\geq 0$ and $q<0$, we have $(\partial P/\partial q)(t,q)>2$, and 
    \begin{equation}\label{equa.6.5}
      \lim_{q\to-\infty}\frac{\partial P}{\partial q}(t,q)=2.
    \end{equation}
  \end{item}
 \end{proposition}
 \begin{proof}
  Fix $t\geq 0$ and $q_0<0$. Then for any $q\leq q_0$, we have $$2^{-n(t-1)}e^{2^{n}q}\leq 2^{-n(t-1)}e^{2^{n}q_0}.$$
  Notice that $$\lim_{n\to\infty}\sqrt[n]{2^{-n(t-1)}e^{2^{n}q_0}}=0<1.$$
   By Weierstrass M-test, we have that the series $\sum_{n=1}^{\infty}2^{-n(t-1)}e^{2^{n}q}$ is uniformly convergent on the set $\{q:q<q_0\}.$
  Hence, we obtain \eqref{equa.6.3}.

  \medskip
  Fix $0<t\leq1$. Since $e^{2^{n}q}\to 1$ $(q\to 0)$, we have 
  \begin{equation*}
    \lim_{q\to0}\frac{\partial P}{\partial q}(t,q)=\frac{\sum_{n=1}^{\infty}2^{-n(t-1)}}{\sum_{n=1}^{\infty}2^{-nt}}.
  \end{equation*}
  Notice that $$\sum_{n=1}^{\infty}2^{-n(t-1)}=\infty, \quad\sum_{n=1}^{\infty}2^{-nt}=\frac{1}{2^t-1}.$$
  Hence, $\eqref{equa.6.4}$ is proved.
  \medskip

In order to prove (2), we need to show 
  $$\lim_{t\to 0}\frac{\sum_{n=1}^{\infty}2^{-n(t-1)}}{\sum_{n=1}^{\infty}2^{-nt}}=\infty.$$ This is analogous to that in \eqref{equa.5.7} and is left to the reader.
  
  \medskip
   Observe that 
  \[\sum_{n=1}^{\infty}2^{-n(t-1)}e^{2^{n}q}>2\sum_{n=1}^{\infty}2^{-nt}e^{2^{n}q}.\]
  As a result,
  \begin{equation*}
    \frac{\partial P}{\partial q}(t,q)>2.
  \end{equation*}
  This proves the first part of (3).
 
\medskip
 Since $\lim_{q\to-\infty}e^{(2^n-2)q}=0$ for $n\geq 2$, it follows that 
\[\lim_{q\to-\infty}\frac{\sum_{n=1}^{\infty}2^{-n(t-1)}e^{2^nq}}{\sum_{n=1}^{\infty}2^{-nt}e^{2^nq}}=\lim_{q\to-\infty}\frac{2^{-(t-1)}e^{2q}(1+\sum_{n=2}^{\infty}2^{-(n-1)(t-1)}e^{(2^n-2)q})}{2^{-t}e^{2q}(1+\sum_{n=2}^{\infty}2^{-(n-1)t}e^{(2^n-2)q})}=2,
\]
which completes the proof.
\end{proof}
\begin{proposition}
   For $(t,q)\in D_1,$
  \begin{equation}\label{equa.6.6}
    \frac{\partial P}{\partial t}(t,q)=\frac{(-\log2)\sum_{n=1}^{\infty}n2^{-n t}e^{2^{n}q}}{\sum_{n=1}^{\infty}2^{-n t}e^{2^{n}q}}.
  \end{equation}
  Consequently,
  \begin{item}
    (1) for fixed $t\geq0$,
    \begin{equation}\label{equa.6.7}
      \lim_{q\to-\infty}\frac{\partial P}{\partial t}(t,q)=-\log2;
    \end{equation}
    \end{item}
    
  \begin{item}
    (2) for $t=1$,
    \begin{equation}\label{equa.6.10}
      \lim_{q\to0}\frac{\partial P}{\partial t}(1,q)=-\frac{\log2}{2}.
    \end{equation}
  \end{item}
\end{proposition}
\begin{proof}
  By Abel's test, it follows that the series $\sum_{n=1}^{\infty}n2^{-n t}e^{2^{n}q}$ converges uniformly with respect to $t$ on the set $D_1$. Hence, we get \eqref{equa.6.6}.

  \medskip
  Fix $t\geq 0$. Write
  $$\frac{\sum_{n=1}^{\infty}n2^{-n t}e^{2^{n}q}}{\sum_{n=1}^{\infty}2^{-n t}e^{2^{n}q}}=\frac{e^{2q}(2^{-t}+\sum_{n=2}^{\infty}n2^{-nt}e^{(2^n-2)q})}{e^{2q}(2^{-t}+\sum_{n=2}^{\infty}2^{-nt}e^{(2^n-2)q})}=\frac{2^{-t}+\sum_{n=2}^{\infty}n2^{-nt}e^{(2^n-2)q}}{2^{-t}+\sum_{n=2}^{\infty}2^{-nt}e^{(2^n-2)q}}.$$
  Notice that for any $n\geq 2$,
  $$\lim_{q\to-\infty}e^{(2^n-2)q}=0.$$
  From the above, it follows that$$\lim_{q\to-\infty}\frac{\sum_{n=1}^{\infty}n2^{-n t}e^{2^{n}q}}{\sum_{n=1}^{\infty}2^{-n t}e^{2^{n}q}}=1,$$
  which proves \eqref{equa.6.7}.
  \medskip

  By \eqref{equa.6.6}, we have $$\frac{\partial P}{\partial t}(1,q)=\frac{(-\log2)\sum_{n=1}^{\infty}n2^{-n}e^{2^{n}q}}{\sum_{n=1}^{\infty}2^{-n}e^{2^{n}q}}.$$
  Since $e^{2^{n}q}\to 1$ ($q\to 0$) for $n\in\N$ and $\sum_{n=1}^{\infty}2^{-n}=1$, we have
  \begin{equation}\label{equa.6.8}
    \lim_{q\to0}\frac{\partial P}{\partial t}(1,q)=(-\log2)\sum_{n=1}^{\infty}n2^{-n}=-\frac{\log2}{2}.
  \end{equation}
\end{proof}

\begin{proposition}\label{prop.4.12}
  Let $\tilde{D}=\{(t,q)\in D:0\leq t\leq 1\}$. For any $\xi\in(2,\infty)$, the system 
  \begin{equation}
    \begin{cases}
     P(t,q)=q\xi,\\
     \dfrac{\partial P}{\partial q}(t,q)=\xi
       \end{cases}
      \end{equation}
   admits a unique solution $\big(t(\xi),q(\xi)\big)\in\tilde{D}$. Further, the functions $t(\xi)$ and $q(\xi)$ are analytic. 
 \end{proposition}

\subsection{Proofs of Theorem \ref{thm.4} (1) and (2)}
It follows from the Birkhoff ergodic theorem that for Lebesgue almost all \x, 
$$\lim_{n\to\infty}\frac{1}{n}\sum_{j=1}^n2^{d_j(x)}=\int 2^{\lceil-\log_2 x\rceil} \,\mathrm{d}x=\infty.$$ Hence,
$$\dim_H E_{\varphi}(\infty)=\dim_H\{x\in(0,1]:\bar{\varphi}(x)=\infty\}=1.$$
 \begin{proposition} The function $t(\xi)$ is strictly incresing on $(2,\infty)$. 
  

\end{proposition}
\begin{proof}
From the analytic area of $P(t,q)$, we know that $q(\xi)\leq 0$. Suppose there exists $\xi_1<\infty$ such that $q(\xi_1)=0$. Then $P(t(\xi_1),0)=0$ by Proposition \ref{prop.4.12}. 
Hence, we have $t(\xi_1)=1$ and $(\partial P/\partial q)(1,0)=\xi_1$. However, by \eqref{equa.6.4}, we have
\[\lim_{(t,q)\to(1,0)}\frac{\partial P}{\partial q}(t,q)=\infty.\]
This contradicts to the hypothesis that $\xi_1<\infty$. Thus $t^{\prime}(\xi)>0$.
\end{proof}

\subsection{Proofs of Theorem \ref{thm.4} (3) and (4)}
\begin{proposition}
  We have
\begin{equation}\label{equa.2}
  \lim_{\xi\to2}t(\xi)=0,\qquad \lim_{\xi\to2}q(\xi)=-\infty.
\end{equation}
Furthermore,
$$\lim_{\xi\to2}t'(\xi)=\infty.$$
\end{proposition}

\begin{proof}To show the first formula in \eqref{equa.2}, we use the same method as in the proof of Proposition \ref{proposition.log}.
  Since $\xi\mapsto t(\xi)$ is increasing and analytic on $(2,\infty)$, we can write its inverse function 
  $\xi=\xi(t)$. Then $\xi^{\prime}(t)>0$ and 
  $$\xi(t)=\frac{P(t,q(t))}{q(t)}=\frac{\partial P}{\partial q}(t,q(t)).$$  
  We conclude from \eqref{equa.6.5} that $$\lim_{q\to-\infty}\frac{\partial P}{\partial q}(0,q)=2,$$  hence that $\lim_{t\to0}\xi(t)=2$, and finally that $\lim_{\xi\to2}t(\xi)=0.$
  \medskip

  Now suppose that there exists a subsequence $\{\xi_{n_k}\}_{k\geq 1}$, such that 
  $$\lim_{k\to\infty}\xi_{n_k}=2, \qquad \lim_{k\to\infty}q(\xi_{n_k})=M>-\infty.$$ Since $\lim_{\xi\to2}t(\xi)=0$, it follows that $\lim_{k\to\infty}t(\xi_{n_k})=0$. By Proposition \ref{prop.4.12}, 
  \begin{equation}\label{contradiction}
    \lim_{k\to\infty}\frac{\partial P}{\partial q}(t(\xi_{n_k}),q(\xi_{n_k}))=\lim_{k\to\infty}\xi_{n_k}=2.
  \end{equation}
  If $M=0$, then by \eqref{equa.6.11}, we have 
  $$\lim_{k\to\infty}\frac{\partial P}{\partial q}(t(\xi_{n_k}),q(\xi_{n_k}))=\lim_{(t,q)\to(0,0)}\frac{\partial P}{\partial q}(t,q)=\infty,$$ which contradicts to \eqref{contradiction}. If $M<0$, then by Proposition \ref{prop.6} (3), we have 
  $$\lim_{k\to\infty}\frac{\partial P}{\partial q}(t(\xi_{n_k}),q(\xi_{n_k}))=\frac{\partial P}{\partial q}(0,M)>2,$$
  which is also a contradiction to \eqref{contradiction}. Hence, $q(\xi)\to -\infty$ $(\xi\to2)$.

  \medskip

Applying the second formula in \eqref{equa.2} and combining \eqref{prop.3} and \eqref{equa.6.7}, we conclude that 
$$\lim_{\xi\to2}t^{\prime}(\xi)=\infty.$$
\end{proof}

\begin{proposition}
 We have
  \begin{equation}\label{equa.6.12}
    \lim_{\xi\to\infty}t(\xi)=1,\quad \lim_{\xi\to\infty}q(\xi)=0.
  \end{equation}
  Furthermore,
  \[\lim_{\xi\to\infty}t'(\xi)=0.\]
 \end{proposition}
 \begin{proof}
  Recall that for any $\xi>2$,
  \begin{equation*}
    \begin{cases}
     P\big(t(\xi),q(\xi)\big)=q(\xi)\xi,\\
     \dfrac{\partial P}{\partial q}\big(t(\xi),q(\xi)\big)=\xi.
       \end{cases}
      \end{equation*}
   Suppose $\big(t(\xi),q(\xi)\big)\to(t_0,q_0)$ as $\xi\to\infty$. The analyticity of $t(\xi)$ and $q(\xi)$ implies that $0<t_0\leq 1$ and $q_0\leq 0$. 
    Suppose $q_0<0$. Then by \eqref{equa.6.3} and \eqref{equa.6.5}, we have $$\lim_{(t,q)\to(t_0,q_0)}\frac{\partial P}{\partial q}(t,q)<\infty,$$ which is a contradiction to $$\lim_{\xi\to\infty}\frac{\partial P}{\partial q}(t,q)=\lim_{\xi\to\infty}\xi=\infty.$$ Thus $q_0=0$. 
    For any $\xi>2$, since $q(\xi)\xi<0$ and $P(t_0,0)\geq 0$, we have $$P(t_0,0)=\lim_{\xi\to\infty}q(\xi)\xi=0.$$ Since $P(1,0)=0$ and $(\partial P/\partial q)(t,q)\to\infty$ as $(t,q)\to(1,0)$, we have $(t_0,q_0)=(1,0)$, which completes the proof of \eqref{equa.6.12}.

    \medskip
By \eqref{equa.6.10} and \eqref{prop.3}, we have $$\lim_{\xi\to\infty}t'(\xi)=0.$$ 
\end{proof}

\medskip
Applying similar arguments to the proof of Theorem \ref{them.similar}  yields  $$\dim_H E_{\varphi}(2)=0.$$

\end{document}